\documentclass
{amsproc}

\usepackage{amssymb}
\usepackage{amsthm}
\usepackage{amsmath}
\usepackage{a4wide}
\usepackage{latexsym}
\usepackage{mathrsfs}
\usepackage{euscript}
\usepackage[v2,tips]{xy}
\usepackage{float}
\usepackage[T1]{fontenc}
\usepackage[latin1]{inputenc}

\newtheorem{theorem}{Theorem}[section]
\newtheorem{lemma}[theorem]{Lemma}
\newtheorem{proposition}[theorem]{Proposition}
\newtheorem{corollary}[theorem]{Corollary}

\theoremstyle{definition}
\newtheorem{definition}[theorem]{Definition}

\newtheorem{remark}[theorem]{Remark}

\numberwithin{equation}{section}

\newcounter{smallromans}

  {\end{list}}

\newcounter{smallalphs}

\newenvironment{alphenumerate}
{\begin{list}{{\normalfont\textrm{(\alph{smallalphs})}}}%
    {\usecounter{smallalphs}\setlength{\itemindent}{0cm}%
      \setlength{\leftmargin}{5.5ex}\setlength{\labelwidth}{5.5ex}%
      \setlength{\topsep}{0.2ex}\setlength{\partopsep}{0ex}%
      \setlength{\itemsep}{0.2ex}}}%
  {\end{list}}

\newcommand{\alphref}[1]{{\normalfont\textrm{(\ref{#1})}}}

\renewcommand{\le}{\ensuremath{\leqslant}}

\renewcommand{\ge}{\ensuremath{\geqslant}}

\newcommand{\N}{\mathbb{N}}

\newcommand{\K}{\mathbb{K}}

\renewcommand{\phi}{\ensuremath{\varphi}}
\renewcommand{\epsilon}{\ensuremath{\varepsilon}}

\begin{document}
\title[Uniqueness of the maximal ideal of operators]{Uniqueness of the
  maximal ideal of operators\\ on the $\ell_p$-sum of
  $\ell_\infty^n\ (n\in\N)$ for $1<p<\infty$} \subjclass[2010]%
{Primary 46B45,
46H10,
47L10;
Secondary 
46B08,
47L20}
\author[T.~Kania]{Tomasz Kania}
\address{Institute of Mathematics, Polish Academy
  of Sciences, ul. \'Sniadeckich 8, 00-956 War\-sza\-wa, Poland}
\email{tomasz.marcin.kania@gmail.com}
\author[N.~J.~Laustsen]{Niels Jakob  Laustsen} 
\address{Department of Mathematics and Statistics, Fylde
  College, Lancaster University, Lancaster LA1 4YF, United Kingdom}
\email{n.laustsen@lancaster.ac.uk}

\begin{abstract} 
  A recent result of Leung (\emph{Proceedings of the American
    Mathematical Society}, to appear) states that the Banach algebra
  $\mathscr{B}(X)$ of bounded, linear operators on the Banach space
  $X=\bigl(\bigoplus_{n\in\N}\ell_\infty^n\bigr)_{\ell_1}$ contains a
  unique maximal ideal. We show that the same conclusion holds true
  for the Banach spaces
  $X=\bigl(\bigoplus_{n\in\N}\ell_\infty^n\bigr)_{\ell_p}$ and
  $X=\bigl(\bigoplus_{n\in\N}\ell_1^n\bigr)_{\ell_p}$ when\-ever
  \mbox{$p\in(1,\infty)$}.
\end{abstract}
\maketitle
\section{Introduction and statement of main results}%
\label{section1}
\noindent 
For $p\in[1,\infty)$, consider the Banach space
\begin{equation*}
  W_p=\biggl(\bigoplus_{n\in\N}\ell_\infty^n\biggr)_{\ell_p}.
\end{equation*}
Denny Leung~\cite{leung} has recently proved that the Banach
algebra~$\mathscr{B}(W_1)$ of all (bounded, linear) operators acting
on~$W_1$ has a unique maximal ideal, thus establishing the dual
version of \cite[Theorem~3.2]{losz}. We shall show that Leung's
conclusion extends to~$\mathscr{B}(W_p)$ for $p\in(1,\infty)$ and
to~$\mathscr{B}(W_p^*)$, where
$W_p^*\cong\bigl(\bigoplus_{n\in\N}\ell_1^n\bigr)_{\ell_q}$ is the
dual Banach space of~$W_p$, with $q\in(1,\infty)$ denoting the
conjugate exponent of~$p$. More precisely, using the following piece
of notation
\begin{equation}\label{defnMX}
  \mathscr{M}_X = \{ T\in\mathscr{B}(X) : \text{the identity operator
    on}\ X\ \text{does not factor through}\ T\}
\end{equation} 
for a Banach space~$X$, we can state our main result as follows.

\begin{theorem}\label{uniquemaxideal}
  For each $p\in (1,\infty)$, the sets $\mathscr{M}_{W_p}$ and
  $\mathscr{M}_{W_p^*}$ given by~\eqref{defnMX} are the unique maximal
  ideals of the Banach algebras $\mathscr{B}(W_p)$ and
  $\mathscr{B}(W_p^*)$, respectively.
\end{theorem}

This theorem adds the spaces~$W_p$ and~$W_p^*$ for $p\in(1,\infty)$ to
the already substantial list, summarized in \cite[p.~4832]{KLjfa}, of
Banach spaces~$X$ for which the set~$\mathscr{M}_X$ is known to be the
unique maximal ideal of~$\mathscr{B}(X)$.

In general, Dosev and Johnson~\cite[p.~166]{dj} observed that, for a
Banach space~$X$, the set~$\mathscr{M}_X$ given by~\eqref{defnMX} is
an ideal of~$\mathscr{B}(X)$ if (and only if) $\mathscr{M}_X$ is
closed under addition, and in the positive case, $\mathscr{M}_X$ is
automatically the unique maximal ideal of~$\mathscr{B}(X)$. Thus, to
prove Theorem~\ref{uniquemaxideal}, it suffices to show that the sets
$\mathscr{M}_{W_p}$ and $\mathscr{M}_{W_p^*}$ are closed under
addi\-tion.

Our approach is completely different from Leung's. Let us here
describe the two most important results that we establish en route to
Theorem~\ref{uniquemaxideal}, as they outline our strategy, and they
may be of some independent interest. First, in
Section~\ref{Sectionnonfixops}, we introduce a new operator ideal in
the following way.  For $p\in[1,\infty]$ and Banach spaces~$X$
and~$Y$, define
\begin{equation}\label{linftynsingops}
  \mathscr{S}_{\{\ell_p^n\: :\: n\in\N\}}(X,Y) =
  \bigl\{T\in\mathscr{B}(X,Y) : T\ \text{does not fix the family}\
  \{\ell_p^n : n\in\N\}\ \text{uniformly}\bigr\}. 
\end{equation}
(Details of this terminology can be found in
Definitions~\ref{defnfixing} and~\ref{defnfixingfamily}.)

\begin{theorem}\label{linftynsingopsIdeal} 
  The class~$\mathscr{S}_{\{\ell_p^n\: :\: n\in\N\}}$ given
  by~\eqref{linftynsingops} is a closed operator ideal in the sense of
  Pietsch for each $p\in[1,\infty]$.
\end{theorem}

Second, in Section~\ref{section3}, we show that the
ideal~$\mathscr{S}_{\{\ell_\infty^n\: :\: n\in\N\}}(W_p)$ is equal to
the set~$\mathscr{M}_{W_p}$.

\begin{theorem}\label{mainthm} Let $p\in(1,\infty)$. An operator
  $T\in\mathscr{B}(W_p)$ fixes the family $\{\ell_\infty^n : n\in\N\}$
  uniformly if and only if the identity operator on~$W_p$ factors
  through~$T$.
\end{theorem}

Ultraproducts play a key role in the proofs of both of these theorems.


\section{Operators fixing certain Banach spaces and the proof of
  Theorem~\ref{linftynsingopsIdeal}}\label{Sectionnonfixops}
\noindent
Throughout this paper, all Banach spaces are supposed to be over the
same scalar field~$\K$, either the real or the complex numbers. By an
\emph{ideal}, we understand a two-sided, algebraic ideal. The term
\emph{operator} means a bounded, linear mapping between Banach
spaces. Given two Banach spaces~$X$ and~$Y$, we
write~$\mathscr{B}(X,Y)$ for the Banach space of all operators
from~$X$ to~$Y$, and we set $\mathscr{B}(X) = \mathscr{B}(X,X)$. 

An operator $T\colon X\to Y$ is \emph{bounded below} by a
constant~$c>0$ if $\|Tx\|\ge c\|x\|$ for each $x\in X$. This is
equivalent to saying that~$T$ is an isomorphism onto its range~$T[X]$,
which is closed, and the inverse operator from~$T[X]$ onto~$X$ has
norm at most~$c^{-1}$.  The class of operators which are bounded below
is open in the norm topology; more precisely, we have the following
estimate, which is an immediate consequence of the sub\-additivity of
the norm.

\begin{lemma}\label{easyfact1}
  Let~$X$ and~$Y$ be Banach spaces, let $c>\epsilon\ge 0$, and let
  $S,T\colon X\to Y$ be operators such that $\| S-T\|\le\epsilon$ and
  $T$ is bounded below by~$c$. Then~$S$ is bounded below
  by~$c-\epsilon$.
\end{lemma}

\begin{definition}\label{defnfixing}
  Let $E$, $X$ and~$Y$ be Banach spaces, let $T\colon X\to Y$ be an
  operator, and let $C\ge 1$. We say that~$T$ \emph{$C$-fixes a copy
    of~$E$} if there is an operator $S\colon E\to X$ of norm at
  most~$C$ such that the composite operator~$TS$ is bounded below
  by~$1/C$. In the case where the value of the constant~$C$ is not
  important, we shall simply say that~$T$ \emph{fixes a copy of~$E$}.

  An operator which does not fix a copy of~$E$ is called
  \emph{$E$-strictly singular;} the set of $E$-strictly singular
  operators from~$X$ to~$Y$ is denoted by~$\mathscr{S}_E(X,Y)$.
\end{definition}

A straightforward application of Lemma~\ref{easyfact1} leads to the
following conclusion.
\begin{corollary}\label{easyfact1cor}
  Let $E$, $X$ and~$Y$ be Banach spaces, let $C'\ge C\ge 1$, and let
  $S,T\colon X\to Y$ be operators such that $T$ $C$-fixes a copy
  of~$E$ and $\| S-T\|\le(C'-C)/C^2C'$. Then~$S$ $C'$-fixes a copy
  of~$E$.
\end{corollary}

It follows in particular that the set~$\mathscr{S}_E(X,Y)$ is
norm-closed in~$\mathscr{B}(X,Y)$ for any Banach spaces~$E$, $X$
and~$Y$. Moreover, the class~$\mathscr{S}_E$ is clearly closed under
arbitrary compositions, in the sense that $STR\in\mathscr{S}_E(W,Z)$
whenever $R\in\mathscr{B}(W,X)$, $T\in\mathscr{S}_E(X,Y)$ and
\mbox{$S\in\mathscr{B}(Y,Z)$} (and~$W$, $X$, $Y$ and~$Z$ are Banach
spaces).  Thus~$\mathscr{S}_E$ is a closed operator ideal in the sense
of Pietsch if (and only if) it is closed under addition. We shall now
show that this is the case provided that the Banach space~$E$ is
\emph{minimal}, in the sense~$E$ is infinite-dimensional and each of
its closed, infinite-dimensional subspaces contains a further subspace
which is isomorphic to~$E$. Examples of minimal Banach spaces include
the classical sequence spaces~$c_0$ and~$\ell_p$ for $1\le p<\infty$
(Pe\l{}czy\'{n}ski~\cite{pelc}), the dual of Tsirelson's space~$T$
(Casazza, Johnson and Tzafriri~\cite{cjt}; note that we follow the
convention, originating from~\cite{fj}, that the term `Tsirelson's
space~$T$' refers to the dual of the space originally constructed by
Tsirelson) and Schlumprecht's space~$S$ (Schlumprecht~\cite{AS}). On
the other hand, we note in passing that Tsirelson's space~$T$ is not
itself minimal~\cite{co}.

We shall require the following lemma (see
\cite[Proposition~2.c.4]{lt1}, where it is attributed to
Kato~\cite{kato}), whose statement involves the following standard piece
of terminology: an operator is \emph{approximable} if it belongs to
the norm-closure of the set of finite-rank operators.

\begin{lemma}\label{katolemma} 
  Let $X$ and $Y$ be infinite-dimensional Banach spaces, and let
  $T\colon X \to Y$ be an operator which is not bounded below on any
  finite-codimensional subspace of~$X$.  Then, for each $\epsilon>0$,
  $X$ contains a closed, infinite-dimensional subspace~$W$ such that
  the restriction of the operator~$T$ to the subspace~$W$ is
  approximable and has norm at most $\epsilon$.
\end{lemma}

\begin{proposition}\label{SEoperatorideal}
  Let~$E$ be a minimal Banach space. Then the class~$\mathscr{S}_E$ of
  $E$-strictly singular operators is a closed operator ideal in the
  sense of Pietsch.
\end{proposition}

\begin{proof}
  By the remarks above, it suffices to show that, for each pair~$X,Y$
  of Banach spaces, the set~$\mathscr{S}_E(X,Y)$ is closed under
  addition. To verify this, suppose that $S\in\mathscr{S}_E(X,Y)$ and
  $T\in\mathscr{B}(X,Y)$ are operators such that
  $S+T\notin\mathscr{S}_E(X,Y)$; we must show that
  \mbox{$T\notin\mathscr{S}_E(X,Y)$}. Choose an operator $R\colon E\to
  X$ such that $(S+T)R$ is bounded below by $c>0$, say. Since
  $S\in\mathscr{S}_E(X,Y)$ and~$E$ is minimal, the restriction of~$SR$
  to any closed, infinite-dimensional subspace of~$E$ is not bounded
  below. Hence Lemma~\ref{katolemma} implies that~$E$ contains a
  closed, infinite-dimensional subspace~$F$ such that $\|SR|_{F}\|\le
  c/2$. After replacing~$F$ with a suitably chosen subspace, we may in
  addition suppose that~$F$ is isomorphic
  to~$E$. Lemma~\ref{easyfact1} shows that $TR|_{F}$ is bounded below
  by~$c/2$, and so $T\notin\mathscr{S}_E(X,Y)$.
\end{proof}

\begin{remark}
  A more general version of Proposition~\ref{SEoperatorideal} can be
  deduced from a result of Stephani~\cite[Theorem~2.1]{stephani}, as
  Rosenberger observed in his Mathematical Review (MR582517) of
  Stephani's paper.
\end{remark}

The connection between Proposition~\ref{SEoperatorideal} and
Theorem~\ref{linftynsingopsIdeal} goes via ultraproducts. We refer the
reader to \cite[Section~11.1]{ak} or \cite[Chapter~8]{DJT} for basic
facts and notation involving ultra\-products. The following lemma is
essentially a quantitative version of the fact that each ultrapower of
a Banach space~$X$ is finitely representable in~$X$.

\begin{lemma}\label{ultrapowerfixinglemma}
  Let~$E$, $X$ and~$Y$ be Banach spaces, where~$E$ is
  finite-dimensional, let \mbox{$C'>C\ge 1$}, let $T\colon X\to Y$ be
  an operator, and let~$\EuScript{U}$ be a free ultrafilter on~$\N$
  such that the ultrapower $T_{\EuScript{U}}\colon X_{\EuScript{U}}\to
  Y_{\EuScript{U}}$ $C$-fixes a copy of~$E$. Then~$T$ $C'$-fixes a
  copy of~$E$.
\end{lemma}

To prove it, we shall require the following simple variant of
\cite[Lemma~11.1.11]{ak}, where we keep record of the constants
involved.

\begin{lemma}\label{AKlemma11.1.11}
  Let $T$ be an operator from a non-zero, finite-dimensional Banach
  space~$E$ into a Banach space~$X$, let~$N$ be a finite
  $\epsilon$-net in the unit sphere of~$E$ for some
  $\epsilon\in(0,1)$, and let $\eta\le\min_{x\in N}\| Tx\|$ and
  $\xi\ge \max_{x\in N}\| Tx\|$. Then
  \[ \frac{\eta - \epsilon(\xi + \eta)}{1-\epsilon}\| x\|\le \|Tx\|\le
  \frac{\xi}{1-\epsilon}\|x\|\qquad (x\in E). \]
\end{lemma}

\begin{proof}[Proof of Lemma~{\normalfont{\ref{ultrapowerfixinglemma}}}.]
  We may suppose that $E$ is non-zero, so that~$E$ has a normalized
  basis $(e_j)_{j=1}^n$; denote by $(f_j)_{j=1}^n$ the corresponding
  coordinate functionals.  Choose \mbox{$C''\in(C,C')$}, and let~$N$ be a
  finite $\epsilon$-net in the unit sphere of~$E$, where
  \[ \epsilon = \frac{C'-C''}{C'(C'')^2\|T\|+C'-C''}\in(0,1). \] By
  the assumption, there is an operator $S\colon E\to X_{\EuScript{U}}$
  of norm at most~$C$ such that the composite
  operator~$T_{\EuScript{U}}S$ is bounded below by~$1/C$. For each
  $j\in\{1,\ldots,n\}$, let
  \mbox{$(x_{j,k})_{k\in\N}\in\ell_\infty(\N,X)$} be a representative
  of the equivalence class of~$Se_j$ in~$X_{\EuScript{U}}$.  Then, for
  each $x\in N$, we have
  \[ \lim_{k,\EuScript{U}}\biggl\|\sum_{j=1}^n\langle x,f_j\rangle
  x_{j,k}\biggr\| = \|S x\|\le C<C''\ \ \text{and}\ \
  \lim_{k,\EuScript{U}}\biggl\|\sum_{j=1}^n\langle x,f_j\rangle
  Tx_{j,k}\biggr\| = \|T_{\EuScript{U}}S x\|\ge \frac{1}{C} >
  \frac{1}{C''}. \] Since~$N$ is finite and~$\EuScript{U}$ is closed
  under finite intersections, the set
  \begin{equation}\label{ultralimits} 
    M = \biggl\{ k\in\N : \biggl\|\sum_{j=1}^n\langle x,f_j\rangle 
    x_{j,k}\biggr\|<C''\ \text{and}\ \biggl\|\sum_{j=1}^n\langle 
    x,f_j\rangle Tx_{j,k}\biggr\| > \frac{1}{C''}\ \ (x\in N)\biggr\}
  \end{equation} 
  belongs to~$\EuScript{U}$, and it is therefore non-empty; choose
  $k\in M$, and define a mapping \mbox{$R\colon E\to X$} by setting
  $Re_j = x_{j,k}$ for each $j\in\{1,\ldots,n\}$ and extending by
  linearity. The estimates given in~\eqref{ultralimits} together with
  Lemma~\ref{AKlemma11.1.11} and the choice of~$\epsilon$ imply that
  $\|R\|\le C''/(1-\epsilon)\le C'$ and~$TR$ is bounded below by \[
  \frac{1/C'' - \epsilon(\|T\|C''+1/C'')}{1-\epsilon} =
\frac{1}{C'}, \] so
  that~$T$ $C'$-fixes a copy of~$E$.
\end{proof}

\begin{definition}\label{defnfixingfamily}
  Let $\mathfrak{F}$ be a non-empty family of Banach spaces.  We say
  that an operator~$T$ \emph{fixes the family~$\mathfrak{F}$
    uni\-form\-ly} if there is a constant~$C\ge 1$ such that~$T$
  $C$-fixes a copy of each Banach space in~$\mathfrak{F}$.
\end{definition}

To state our next result concisely, it is convenient to introduce the
notation $E_p = \ell_p$ for $p\in[1,\infty)$ and $E_\infty = c_0$.

\begin{corollary}\label{fixinglemma}
Let $X$ and $Y$ be Banach spaces, let $T\in\mathscr{B}(X,Y)$, and let
$p\in[1,\infty]$. Then the following three conditions are equivalent:
  \begin{alphenumerate}
  \item\label{fixinglemma1} the operator $T$ fixes the family
    $\{\ell_p^n : n\in\N\}$ uniformly;
  \item\label{fixinglemma2} for every free ultrafilter~$\EuScript{U}$
    on~$\N$, the ultrapower $T_{\EuScript{U}}\colon X_{\EuScript{U}}\to
    Y_{\EuScript{U}}$ fixes a copy of~$E_p;$
  \item\label{fixinglemma3} there exists a free
    ultrafilter~$\EuScript{U}$ on~$\N$ such that the ultrapower
    $T_{\EuScript{U}}\colon X_{\EuScript{U}}\to Y_{\EuScript{U}}$ fixes
    the family $\{\ell_p^n : n\in\N\}$ uniformly.
\end{alphenumerate}
\end{corollary}

\begin{proof}
  \alphref{fixinglemma1}$\Rightarrow$\alphref{fixinglemma2}. Suppose
  that there exists a constant~$C\ge 1$ such that, for each $n\in\N$,
  we can find an operator $S_n\colon\ell_p^n\to X$ of norm at most~$C$
  such that the composite operator $TS_n$ is bounded below by~$1/C$,
  and let~$\EuScript{U}$ be a free ultrafilter on~$\N$. Then we have
  an operator $S = (\prod S_n)_{\EuScript{U}}$ of norm at most~$C$
  from the ultraproduct $(\prod\ell_p^n)_{\EuScript{U}}$ into the
  ultrapower~$X_\EuScript{U}$, and the composite operator
  $T_{\EuScript{U}}S$ is bounded below by~$1/C$.  For each~$n\in\N$,
  $\ell_p^n$ is an $L_p(\mu)$-space for $p<\infty$ and a $C(K)$-space
  for $p=\infty$, and these classes are preserved by ultraproducts
  (see, \emph{e.g.}, \cite[Theorem~8.7]{DJT}). Thus the domain of~$S$
  is an in\-finite-dimensional $L_p(\mu)$-space for $p<\infty$ and an
  in\-finite-dimensional $C(K)$-space for $p=\infty$, so that in
  either case it contains an isomorphic copy of~$E_p$. Taking an
  operator $R\colon E_p\to (\prod\ell_p^n)_{\EuScript{U}}$ which is
  bounded below, we see that $T_{\EuScript{U}}SR$ is also bounded
  below, so that~$T_{\EuScript{U}}$ fixes a copy of~$E_p$.

  The implication
  \alphref{fixinglemma2}$\Rightarrow$\alphref{fixinglemma3} is
  obvious, while
  \alphref{fixinglemma3}$\Rightarrow$\alphref{fixinglemma1} follows
  from Lemma~\ref{ultrapowerfixinglemma}.
\end{proof}

\begin{proof}[Proof of Theorem~{\normalfont{\ref{linftynsingopsIdeal}}}]
  The class~$\mathscr{S}_{\{\ell_p^n\: :\: n\in\N\}}$ is clearly
  closed under arbitrary compositions and contains all finite-rank
  operators, while Corollary~\ref{easyfact1cor} shows that it is
  closed in the operator norm. Now suppose that
  $S,T\in\mathscr{S}_{\{\ell_p^n\: :\: n\in\N\}}(X,Y)$ for some Banach
  spaces~$X$ and~$Y$. Corollary~\ref{fixinglemma} implies that
  $S_{\EuScript{U}},T_{\EuScript{U}}\in
  \mathscr{S}_{E_p}(X_{\EuScript{U}},Y_{\EuScript{U}})$ for every free
  ultrafilter~$\EuScript{U}$ on~$\N$, where $E_p = \ell_p$ for
  $p<\infty$ and $E_p = c_0$ for $p=\infty$.  Consequently, we have
  \mbox{$(S+T)_{\EuScript{U}} = S_{\EuScript{U}}+T_{\EuScript{U}}\in
    \mathscr{S}_{E_p}(X_{\EuScript{U}},Y_{\EuScript{U}})$} by
  Proposition~\ref{SEoperatorideal}, and hence another application of
  Corollary~\ref{fixinglemma} shows that
  $S+T\in\mathscr{S}_{\{\ell_p^n\: :\: n\in\N\}}(X,Y)$.
\end{proof}

\section{The proofs of Theorems~\ref{mainthm}
  and~\ref{uniquemaxideal}}\label{section3}
\noindent
We begin by establishing some lemmas and introducing some notation
that will be required in the proof of Theorem~\ref{mainthm}. Our first
lemma needs no proof: it follows immediately from the $1$-injectivity
of the Banach space~$\ell_\infty^n$.

\begin{lemma}\label{easyfact2}
  Let $n\in\N$, let $X$ be a Banach space, and let $T\colon
  \ell_\infty^n\to X$ be an operator which is bounded below
  by~$c>0$. Then $T$ has a left inverse~$X\to\ell_\infty^n$ of
  norm at most~$c^{-1}$. 
\end{lemma}

Our second lemma concerns strictly singular perturbations of operators
that fix~$\ell_p$ for some $p\in[1,\infty)$ or~$c_0$.

\begin{lemma}\label{ssperturbation}
  Let $X$ and $Y$ be Banach spaces, let $E = \ell_p$ for some
  $p\in[1,\infty)$ or $E = c_0$, let $C'>C\ge 1$, and let $S,T\colon
    X\to Y$ be operators, where $S$ is strictly singular and~$T$
    $C$-fixes a copy of~$E$. Then $S+T$ $C'$-fixes a copy of~$E$.
\end{lemma}
\begin{proof}
  By the assumption, we can choose an operator $R\colon E\to X$ such
  that $\|R\|\le C$ and $TR$ is bounded below by~$1/C$.  Set $\epsilon
  = (C'-C)/C'(C+1)\in(0,1)$.  Since~$SR$ is strictly singular,
  Lemma~\ref{katolemma} implies that $E$ contains a closed,
  in\-finite-di\-men\-sional subspace~$F$ such that
  \mbox{$\|SR|_F\|\le\epsilon$}.  Keeping careful track of the
  constants in the proof of Pe\l{}czy\'{n}ski's theorem that~$E$ is
  minimal, as it is given in \cite[Proposition~2.2.1]{ak}, for
  instance, as well as in the proof of \cite[Theorem~1.3.9]{ak}, we
  see that in fact every closed, infinite-dimensional sub\-space
  of~$E$ contains almost isometric copies of~$E$.  We can therefore
  find an operator $U\colon E\to F$ such that \mbox{$(1-\epsilon)\|
    x\|\le\|Ux\|\le\|x\|$} for each $x\in E$. Hence we have
  $\|RU\|\le\|R\|<C'$, \[\|SRU\|\le
  \|SR|_F\|\,\|U\|\le\epsilon\qquad\text{and}\qquad \|TRUx\|\ge
  \frac{1}{C}\|Ux\|\ge\frac{1-\epsilon}{C}\|x\|\qquad (x\in E), \] so
  that $(S+T)RU$ is bounded below by $(1-\epsilon)/C - \epsilon =
  1/C'$ by Lemma~\ref{easyfact1} and the choice of~$\epsilon$. This
  shows that $S+T$ $C'$-fixes a copy of~$E$.
\end{proof}

We shall next introduce some notation and terminology related to
Banach spaces of the form
\begin{equation}\label{defnX} X =
  \biggl(\bigoplus_{n\in\N} X_n\biggr)_{\ell_p} = \biggl\{ (x_n)_{n\in\N}:
  x_n\in X_n\ (n\in\N)\ \text{and}\
  \sum_{n=1}^\infty\|x_n\|^p<\infty\biggr\}, 
\end{equation} 
where $(X_n)_{n\in\N}$ is a sequence of Banach spaces and
$p\in[1,\infty)$. For each $n\in\N$, we write $\iota_n\colon X_n\to X$
and $\pi_n\colon X\to X_n$ for the canonical $n^{\text{th}}$
coordinate embedding and projection, respectively.  Given an
operator~$T$ on~$X$, we associate with it the $(\N\times\N)$-matrix
$(T_{j,k})$, where $T_{j,k} = \pi_jT\iota_k\colon X_k\to X_j$ for each
pair $j,k\in\N$. We say that $T$ has \emph{finite rows} if, for each
$j\in\N$, there exists $k_0\in\N$ such that $T_{j,k} = 0$ whenever
$k>k_0$, and that~$T$ has \emph{finite columns} if, for each $k\in\N$,
there exists $j_0\in\N$ such that $T_{j,k} = 0$ when\-ever $j>j_0$.

The following elementary perturbation result is a special case of
\cite[Lemma~2.7]{LLR}.
\begin{lemma}\label{finiterowsandcols}
  Let $T$ be an operator on a Banach space~$X$ of the
  form~\eqref{defnX}, where $X_n$ is fi\-nite-di\-men\-sional for each
  $n\in\N$ and $p\in(1,\infty)$. Then, for each $\epsilon>0$, there
  exists an operator \mbox{$T'\in\mathscr{B}(X)$} with finite rows and
  finite columns such that the operator $T-T'$ is approximable and has
  norm at most~$\epsilon$.
\end{lemma}

Set $P_0 = 0$ and $P_n = \sum_{j=1}^n \iota_j\pi_j$ for $n\in\N$. We
can then state our final lemma as follows.
\begin{lemma}\label{easyfact3}
  Let~$X$ be a Banach space of the form~\eqref{defnX}, let \mbox{$0\le
    k_1<k_1'\le k_2<k_2'\le\cdots$} be an increasing sequence of
  integers, and let $(R_n\colon X_n\to X)_{n\in\N}$ and $(S_n\colon
  X\to X_n)_{n\in\N}$ be uniformly bounded sequences of
  operators. Then
  \begin{equation}\label{DEfnRS} 
    R\colon\ (x_n)_{n\in\N}\mapsto\sum_{n=1}^\infty
    (P_{k_n'}-P_{k_n})R_nx_n\qquad\text{and}\qquad S\colon x\mapsto
    (S_n(P_{k_n'}-P_{k_n})x)_{n\in\N} 
  \end{equation} 
  define operators on~$X$ of norms at most $\sup_{n\in\N}\|R_n\|$ and
  $\sup_{n\in\N}\|S_n\|$, respectively.
\end{lemma}

\begin{proof} Set $C_1=\sup_{n\in\N}\|R_n\|$ and
  $C_2=\sup_{n\in\N}\|S_n\|$, and let $x = (x_n)_{n\in\N}\in X$ be
  given. We must show that the proposed definitions~\eqref{DEfnRS} of
  $Rx$ and $Sx$ belong to~$X$ and have norms at most $C_1\|x\|$ and
  $C_2\|x\|$, respectively; the result will then follow because~$R$
  and~$S$ are easily seen to be linear. The required estimate for~$S$
  is straightforward:
  \[ \sum_{n=1}^\infty \|S_n(P_{k_n'}-P_{k_n})x\|^p\le
  C_2^p\sum_{n=1}^\infty \|(P_{k_n'}-P_{k_n})x\|^p\le C_2^p\|x\|^p. \]
  Concerning~$R$, we define $y_j\in X_j$ for each $j\in\N$ as follows:
  $y_j = \pi_jR_nx_n$ if $k_n<j\le k_n'$ for some (necessarily unique)
  $n\in\N$, and $y_j = 0$ otherwise. Then, for each $m\in\N$, we have
  \[ \sum_{j=k_m+1}^\infty \|y_j\|^p =
  \sum_{n=m}^\infty\sum_{j=k_m+1}^{k_m'} \|\pi_jR_nx_n\|^p =
  \sum_{n=m}^\infty \|(P_{k_n'}-P_{k_n})R_nx_n\|^p\le
  C_1^p\|(I_X-P_{m-1})x\|^p. \] Taking $m=1$, we see that $y$ belongs
  to~$X$ with norm at most $C_1\|x\|$. Moreover, we deduce that the
  series $\sum_{n=1}^\infty (P_{k_n'}-P_{k_n})R_nx_n$ is convergent
  with sum~$y$ because
  \[ \biggl\|y - \sum_{n=1}^m (P_{k_n'}-P_{k_n})R_nx_n\biggr\|^p =
  \sum_{j=k_{m+1}+1}^\infty \|y_j\|^p\le C_1^p\|(I_X-P_m)x\|^p\to
  0\quad\text{as}\quad m\to\infty, \] so that $Rx = y$, and the
  conclusion follows.
\end{proof}

\begin{proof}[Proof of Theorem~{\normalfont{\ref{mainthm}}}]
  The implication $\Leftarrow$ is easy to verify. Suppose that
  $I_{W_p} = STR$ for some operators $R,S\in\mathscr{B}(W_p)$, and let
  $C = \sqrt{\|R\|\,\|S\|}$.  By replacing~$R$ and $S$ with $CR/\|R\|$
  and $CS/\|S\|$, respectively, we may suppose that $\|R\| = \|S\| =
  C$. Then, for each $n\in\N$, the composite operator
  $TR\iota_n\colon\ell_\infty^n\to W_p$ is bounded below by~$1/\|S\| =
  1/C$ and \mbox{$\| R\iota_n\|\le\|R\| = C$}, so that~$T$
  $C$-fixes~$\ell_\infty^n$.

  Conversely, suppose that~$T$ fixes the family $\{\ell_\infty^n :
  n\in\N\}$ uniformly.  We may without loss of generality suppose that
  $\|T\|=1$.  Take a free ultrafilter~$\EuScript{U}$
  on~$\N$. Corollary~\ref{fixinglemma} shows that the
  ultrapower~$T_{\EuScript{U}}$ $C$-fixes a copy of~$c_0$ for some
  $C\ge 1$. Choose constants $C_1>C_2>C_3>C_4>C$, and set $\epsilon =
  \min\{(C_4-C)/C^2C_4,1/C_1^{2}\}\in(0,1)$. By
  Lemma~\ref{finiterowsandcols}, we can find an operator
  $T'\in\mathscr{B}(W_p)$ with finite rows and columns such that $\|T
  - T'\|<\epsilon/2$. Set $T'' = T'/\|T'\|$. Since \[ \|
  T_{\EuScript{U}} - T''_{\EuScript{U}}\| = \|T-T''\|\le \|T-T'\| +
  \biggl\|\biggl(1 - \frac{1}{\|T'\|}\biggr)T'\biggr\| = \|T-T'\| +
  \bigl|\|T'\|-\|T\|\bigr| <\epsilon\le \frac{C_4-C}{C^2C_4}, \]
  Corollary~\ref{easyfact1cor} implies that $T''_{\EuScript{U}}$
  $C_4$-fixes a copy of~$c_0$.

  By induction, we shall construct sequences $0 = k_0 = k_0'\le
  k_1<k_1'\le k_2<k_2'\le\cdots$ and $0=m_0<m_1<m_2<\cdots$ of
  integers and sequences $(R_n\colon\ell_\infty^n\to W_p)_{n\in\N_0}$
  and \mbox{$(S_n\colon W_p\to\ell_\infty^n)_{n\in\N_0}$} of
  operators, each having norm at most~$C_1$, such that
  \begin{equation}\label{inductionclaim1}
    (I_{W_p} - P_{m_n})T''P_{k_n'} = 0 =
    P_{m_{n-1}}T''(I_{W_p}-P_{k_n}) \end{equation} and the diagram
  \begin{equation}\label{inductionclaim2} 
\spreaddiagramcolumns{-3ex}%
    \xymatrix{%
      & \ell_\infty^n\ar^-{\displaystyle{I_{\ell_\infty^n}}}[rr]%
      \ar_-{\displaystyle{R_n}}[ld]%
      &&  \ell_\infty^n\\
      W_p\ar_-{\displaystyle{P_{k_n'} - P_{k_n}}}[rd] &&&&
      W_p\ar_-{\displaystyle{S_n}}[lu]\\
      & W_p\ar^-{\displaystyle{T''}}[rr] &&
      W_p\ar_-{\displaystyle{P_{m_n} - P_{m_{n-1}}}}[ru]} \end{equation} is
  commutative for each $n\in\N$.

  The only reason that we have included the case $n=0$ is that it
  makes the start of the induction trivial (whereas if we began with
  $n=1$, we would need to carry out a small amount of checking, which
  would duplicate parts of the induction step).  Indeed, we can simply
  take $R_0 = S_0 =0$ (as well as $k_0 = k_0' = m_0 = 0$, as already
  stated).

  Now assume that, for some $N\in\N_0$, integers $0 = k_0 = k_0'\le
  k_1<k_1'\le\cdots\le k_N<k_N'$ and $0=m_0<m_1<\cdots<m_N$ and
  operators $(R_n\colon\ell_\infty^n\to W_p)_{n=0}^N$ and $(S_n\colon
  W_p\to\ell_\infty^n)_{n=0}^N$ of norms at most~$C_1$ have been
  chosen in accordance
  with~\eqref{inductionclaim1}--\eqref{inductionclaim2}. Since~$T''$
  has finite rows, we can choose $k_{N+1}\ge k_N'$ such that
  $T''_{r,s} = 0$ whenever $1\le r\le m_N$ and $s>k_{N+1}$. Then we
  have $P_{m_N}T''(I_{W_p}-P_{k_{N+1}}) = 0$. For convenience, set
  $T''_{N+1} = (I_{W_p} - P_{m_N})T''(I_{W_p}-P_{k_{N+1}})$. This is a
  finite-rank perturbation of~$T''$, and
  consequently~$(T''_{N+1})_{\EuScript{U}}$ is a finite-rank
  perturbation of~$T''_{\EuScript{U}}$ because ultrapowers of
  finite-rank operators have finite rank. Hence
  Lemma~\ref{ssperturbation} implies that~$(T''_{N+1})_{\EuScript{U}}$
  $C_3$\nobreakdash-fixes a copy of~$c_0$, and thus
  of~$\ell_\infty^{N+1}$. This, in turn, means that~$T''_{N+1}$
  $C_2$-fixes a copy of~$\ell_\infty^{N+1}$ by
  Lemma~\ref{ultrapowerfixinglemma}; that is, we can find an operator
  $R_{N+1}\colon\ell_\infty^{N+1}\to W_p$ of norm at most~$C_2$ such
  that $T''_{N+1}R_{N+1}$ is bounded below by~$1/C_2$. The fact
  that~$R_{N+1}$ has finite rank means that we can take
  $k_{N+1}'>k_{N+1}$ such that $\|(I_{W_p}-P_{k_{N+1}'})R_{N+1}\|\le
  1/C_2 - 1/C_1$. Lemma~\ref{easyfact1} then implies that $(I_{W_p} -
  P_{m_N})T''(P_{k_{N+1}'}-P_{k_{N+1}})R_{N+1}$ is bounded below
  by~$1/C_1$. Since~$T''$ has finite columns, we can choose
  $m_{N+1}>m_N$ such that $T''_{r,s}=0$ whenever $r>m_{N+1}$ and $1\le
  s\le k_{N+1}'$. Then we have $(I_{W_p} - P_{m_{N+1}})T''P_{k_{N+1}'}
  = 0$, and consequently \[ (P_{m_{N+1}} -
  P_{m_N})T''(P_{k_{N+1}'}-P_{k_{N+1}})R_{N+1} = (I_{W_p} -
  P_{m_N})T''(P_{k_{N+1}'}-P_{k_{N+1}})R_{N+1}, \] which is bounded
  below by~$1/C_1$, so Lemma~\ref{easyfact2} gives an operator
  $S_{N+1}\colon W_p\to\ell_\infty^{N+1}$ of norm at most~$C_1$ such
  that the diagram~\eqref{inductionclaim2} commutes for $n=N+1$.
  Hence the induction continues.

  As in Lemma~\ref{easyfact3}, we can now define operators $R,S\colon
  W_p\to W_p$ of norms at most~$C_1$ by
  \[ Rx = \sum_{n=1}^\infty (P_{k_n'} - P_{k_n})R_nx_n\quad
  \text{and}\quad Sx = (S_n(P_{m_n} - P_{m_{n-1}})x)_{n\in\N}\quad (x =
  (x_n)_{n\in\N}\in W_p). \] Then, for each $r,s\in\N$, we have
  \[ \pi_r(ST''R)\iota_s(x) = S_r(P_{m_r} - P_{m_{r-1}})T''(P_{k_s'} -
  P_{k_s})R_sx = \begin{cases} x\ &\text{if}\ r=s\\ 0\
    &\text{otherwise}\end{cases}\qquad (x\in\ell_\infty^s) \]
  by~\eqref{inductionclaim1}--\eqref{inductionclaim2}, and therefore
  $ST''R = I_{W_p}$. Since
  \[ \|STR - I_{W_p}\|\le \|S\|\,\|T-T''\|\, \|R\|<C_1^2\epsilon\le
  1 \] by the choice of~$\epsilon$, we conclude that the operator
  $STR$ is invertible, and the result follows.
\end{proof}

\begin{proof}[Proof of Theorem~{\normalfont{\ref{uniquemaxideal}}}]
  Theorem~\ref{mainthm} shows that~$\mathscr{M}_{W_p} =
  \mathscr{S}_{\{\ell_\infty^n\: :\: n\in\N\}}(W_p)$, which is an
  ideal by Theorem~\ref{linftynsingopsIdeal}, and it is therefore the
  unique maximal ideal of~$\mathscr{B}(W_p)$ by the observation of
  Dosev and Johnson that was stated in the Introduction.

  The Banach space~$W_p$ is reflexive because $p\in(1,\infty)$. Hence
  the mapping $T\mapsto T^*$, which maps an operator~$T$ to its
  adjoint~$T^*$, is a linear, anti-multiplicative, isometric bijection
  of the Banach algebra~$\mathscr{B}(W_p)$
  on\-to~$\mathscr{B}(W_p^*)$, and so it induces an order isomorphism
  between the lattices of ideals of these two Banach algebras. In
  particular, the image under this mapping of the unique maximal
  ideal~$\mathscr{M}_{W_p}$ of~$\mathscr{B}(W_p)$ is the unique
  maximal ideal of~$\mathscr{B}(W_p^*)$, and this ideal is given by
  \[ \{T^* : I_{W_p}\neq STR\ (R,S\in\mathscr{B}(W_p))\} = \{T^* :
  I_{W_p^*}\neq R^*T^*S^*\ (R,S\in\mathscr{B}(W_p))\} =
  \mathscr{M}_{W_p^*}. \qedhere\]
\end{proof}

\section*{Acknowledgements} 
\noindent The research on which this paper is based was initiated at
the Fields Institute in Toronto, Canada, while both authors were
participating in the \emph{Thematic Program on Abstract Harmonic
  Analysis, Banach and Operator Algebras}.  We are grateful to the
lead and session organizers (H.~G.~Dales, G.~A.~Elliott, A.~T.--M.~Lau
and M.~Neufang) for the invitation to take part and the financial
support that we received, and to everyone at Fields for their kind
hospitality, which made our stay enjoyable, stimulating and
productive.

The first author was supported by a postdoctoral fellowship from the
Warsaw Center of Mathematics and Computer Science, while Lancaster
University supported the second author's travel. We gratefully
acknowledge this support.

\bibliographystyle{amsplain}

\end{document}